\documentclass{article}
\usepackage{authblk}
\usepackage[english]{babel}
\usepackage{amsthm}
\usepackage[letterpaper,top=2cm,bottom=2cm,left=3cm,right=3cm,marginparwidth=1.75cm]{geometry}
\usepackage[lf,enc=t1]{berenis}
\usepackage{mathpazo}
\usepackage{amsmath}

\newtheorem{defi}{Definition}[section]
\newtheorem{theo}{Theorem}[section]
\newtheorem{lemma}{Lemma}[section]
\newtheorem{prop}{Proposition}[section]

\makeatletter
\newcommand{\subjclass}[2][2020]{%
  \let\@oldtitle\@title%
  \gdef\@title{\@oldtitle\footnotetext{#1 \emph{Mathematics subject classification.} #2}}%
}
\newcommand{\keywords}[1]{%
  \let\@@oldtitle\@title%
  \gdef\@title{\@@oldtitle\footnotetext{\emph{Key words and phrases.} #1.}}%
}
\makeatother
\title{Existence result and free boundary limit of a tumor growth  model with necrotic core}
\author{ Samiha Belmor}

\AtEndDocument{\bigskip{\footnotesize%
  
  \textit{E-mail address}, Samiha  BELMOR: \texttt{belmor.samiha@gmail.com} \par

}}

\begin{document}

\subjclass { Primary: 35B45, 35B65, 35D30, 35Q92}
\keywords { Existence;  incompressible limit; free Boundary problem; two cell population; rate of convergence; necrotic.}
\maketitle
\begin{abstract}
We analyze a system of cross-diffusion equations that models the growth of an avascular-tumor spheroid. The model incorporates two nonlinear diffusion effects, degeneracy type and super diffusion. We prove the global existence of weak solutions and justify the convergence towards the free boundary problem of the Hele-Shaw type when the pressure  gets stiff.  We also investigate the convergence rate of the solutions in $L^1-$Lebesgue spaces.

\end{abstract}

\section{Introduction}

We consider a compressible model of tumor growth with initial data. It takes the following form 
\begin{align}
\partial_t n_l- div(n_l \nabla p) &=G(c) n_l-K_{D}(c) n_l, \qquad \text { in } Q_{T} \equiv Q \times(0, T), \label{Eq1} \\
\partial_t n_d - div( n_d\nabla p) &=K_{D}(c) n_l-\mu n_d, \qquad \text { in } Q_{T} \equiv Q \times(0, T), \label{Eqq1}\\
\partial_t c&=d \Delta c-f(c) n_l, \qquad \text { in } Q_{T},   \\
n_{1} \nabla p \cdot \mathbf{n} &=n_{2} \nabla p \cdot \mathbf{n}=0 \qquad \text { on } \Sigma_{T} \equiv \partial Q \times(0, T), \\
c &=c_{\infty} \text { on } \Sigma_{T}, \\
(n_l|_{t=0}, n_d|_{t=0}, c|_{t=0})&=(n_{l}^0, n_{d}^0, c^0), \label{Eq2}
\end{align}
where $Q$  is a bounded domain in $\mathbb{R}^d$ with Lipschitz boundary $\partial Q$, $T$ is a positive number, $\mathbf{n}$ is the unit outward normal to $\partial Q $ and
\begin{equation}\label{Press}
n=n_{l}+n_{d}, \quad P= p_\kappa=\kappa \frac{n}{1-n}, \quad \kappa \geq 0 ,
\end{equation}
 We assume there is a continuous motion of cells within the tumor, we indicate this movement by the velocity fields $\Vec{v}$ so that by Darcy’s law \cite{Pert}, we have 
$$\Vec{v}= -\nabla P  \quad \text{in} \  \ Q, \quad  \kappa >0.$$

This problem was proposed by Ward and King  in \cite{Ward1, Ward2}, and modified later by Shangbin \cite{Cui} which assumes that the tumor is  ball-shaped and all the functions are  radially symmetric in space. In the model the cells
are classified in tow phases: $n_{l}$ refer to live tumor cells and $n_{d}$ refer to dead tumor cells, $c$ represents the concentration of nutrients that the live cells receive from its boundary, $d$ is the diffusion coefficient of nutrient  which is supposed to be a positive constant.  $G(c)$ is the growth rate of tumor cells when nutrient supplement  is at level $c$, $K_D(c)$ is the death rate of tumor cells when nutrient supplement  at level $c$, the dead cells are removed at rate $\mu$ which is a positive constant independent
of $c$. When the oxygen and nutrients are insufficient in the central regions due to the lack of the vessel formation, the cell proliferation rate decrease under the activity of cell killing agents, the inner core of the tumor  will therefore enter into a necrotic state. Note that necrotic is not a reversible process, which means dead cells are never be live cells again, and only live cells consume nutrients with the consumption rate $f(c)$.  Mathematically we have the following  biological assumptions on the growth rate and the reaction terms, which  are quite similar to the one stated in \cite{Cui}.
\begin{itemize}
\item (A1) $f \in C^{1}[0, \infty),$ $f'>0$, and $f(0)=0$.
\item (A2)  $G \in C^{1}[0, \infty),$ $G' >0$, and $G(0)=0$.
\item (A3) $K_{D} \in C^{1}[0, \infty),$ $K_D'<0$, $K_{D}(c) \geq 0$ for $c \geq 0$.
\item (A 4) $ \mu>K_{D}(0)>0$.
\end{itemize}
The condition (A4) means that when there is no nutrient the dead core may not necessarily increase. 
We shall not call the expressions of the functions $f(c)$, $G(c)$ and $K_D(c)$ here, for the reader references, we refer to \cite{Ward1, Ward2}.

Continuum mechanical models are used in the literature to describe the mechanical properties of tissue growth, different models describing tumors consisting of different kinds of cells, we mention, Bertsch  \textit{et al.}  \cite{Ber} considered model for healthy and tumor cells.  Jonathan \textit{et al.} \cite{Pett} constructed a model for proliferating cells and quiescent cells.  Model of non-necrotic tumors was given by Byrne \textit{et al.} in \cite{Byr}. 
One can also model tissue growth  by considering free
boundary models \cite{Avn, Avn2} where  the tumor  treated as an expanding domain in (t) and describe its movement and shape by the motion of the boundary. A specific type of convergence established via the so-called incompressible limit has been used to draw a connection between these two types of models.

The rigorous justification of the incompressible limit has been studied vastly in different contexts relying on the generality of the system, the type of the system ( Navier-Stokes model \cite{Vauc}, Cahn-Hilliard type model \cite{Ebe}, Cahn-Hilliard-Brinkman type model \cite{Ebe},  Keller-Segel model \cite{Elb},  Patlak-Keller-Segel model \cite{Pat} or Parabolic type model \cite{Pert}), the modeling context (Darcy's law, or Brinkman's law \cite{Pert2}), the pressure law (power-law \cite{Liu}, or singular law \cite{Hech2} ) as well as the type of the initial data considered (well-prepared or ill-prepared). The mathematical justification of  the incompressible limit was first initiated by Bénilan and Crandall \cite{Phi} for the filtration equations $\partial_t n =\Delta \varphi$ where $\varphi(s)= s^{\gamma}$ with ill-prepared initial data.  This work are then extended by several authors in different settings, for recent work we mention  \cite{Dav, Liu, CM, Debi, Debii}. The limit model in each case is a Hele-Shaw type free boundary model.
 
 The core of the analysis in this paper is to study the existence assertion for \eqref{Eq1}-\eqref{Eq2}, and establish some uniform  estimates in order to investigate the vanishing limit of solutions as $\kappa \to 0$. Moreover, to the best of our knowledge, we are the first to prove the direct $L^1-$ convergence rate of the solution $(n_0, c_0).$

 We note that by combining \eqref{Eq1} and \eqref{Eqq1}, we obtain the total density $n$ equation, which solves the problem
 \begin{equation}\label{eq-density}
     \partial_t n_{\kappa}- \Delta(H_\kappa(n_{\kappa})) =G(c_{\kappa}) n_{l, \kappa}-\mu n_{d, \kappa}\equiv R_{\kappa} \qquad \text { in }  \ Q_{T}.
 \end{equation}
where 
\begin{equation}\label{H-equation}
H_\kappa \left(n_{\kappa}\right):=\int_{0}^{n_{\kappa}} s p_{\kappa}^{\prime}(s) \mathrm{d} s= p_{\kappa}+\kappa \ln \left(1-n_{\kappa}\right) .
\end{equation}
As a result, if we formally set $\kappa \to 0$, we will encounter the following problem 
$$
\begin{aligned}
\partial_{t} n_{0}-\Delta p_{0} &=G(c_0) n_{l,0}- \mu n_{d,0} \equiv R_{0} \quad \text { in } \mathcal{D'}(Q_{T}), \\
\partial_t n_{l,0}- div(n_{l,0} \nabla p_0) &=G(c_0) n_{l,0} -K_{D}(c_0) n_{l,0}, \qquad \text { in } \mathcal{D'}(Q_{T}) \\
\partial_t n_{d,0} - div( n_{d,0}\nabla p_0) &=K_{D}(c_0) n_{l,0}-\mu n_{d,0}, \qquad \text { in }  \mathcal{D'}(Q_{T})\\
\partial_t c_0 &=d \Delta c_0 -f(c_0) n_{l,0} , \qquad \text { in } Q_{T}  \\
n_{l,0} \nabla p_0 \cdot \mathbf{n} &=n_{d,0} \nabla p_0 \cdot \mathbf{n}=0 \qquad \text { on } \Sigma_{T} \equiv \partial Q \times(0, T), \\
\nabla p_{0} \cdot \mathbf{n} &=0 \text { on } \Sigma_{T}, \\
c &=c_{\infty} \text { on } \Sigma_{T}, \\
n_{0}(x, 0) &=n_{l,0}+ n_{d,0} \text { on } Q .
\end{aligned}
$$

Multiplying the total density equation \eqref{eq-density}  by $p'(n_{\kappa})$ we find an equation of the pressure
\begin{equation}\label{eq-pressur}
    \partial_t p_{\kappa} -(\frac{p_{\kappa}}{\kappa} +p_{\kappa})\Delta p_{\kappa} -|\nabla p_{\kappa}|^2 = \frac{1}{\kappa}(p_{\kappa} + \kappa)^2(n_{l, \kappa} G(c_{\kappa}) -\mu n_{d, \kappa}).
\end{equation}
Passing formally to the limit $\kappa \to 0$ into \eqref{eq-pressur}, we can reach the so-called complementarity relation. 
\begin{equation}\label{Com-relation}
    -p_{0}^2 \Delta p_0 = p_0^2(n_{l,0}G(c_0)-\mu n_{d,0}).
\end{equation}
It is not difficult to derive from \eqref{Press}  that when $\kappa$ goes to $0$, we expect having the relation
$(1-n_0)p_0=0$ in the sens of distribution. This relation leads to decomposition of the domain in two parts. $P(t) =\{x, \ \ p(t,x)>0\}$ and its complementary $P^c(t)=\{x, \ \ p(t,x)=0\}$. Unlike the complementarity relation that was established in \cite{Degond},  the authors were unable to recover the usual relation due to the lack of compactness  in time, nevertheless, using the Aubin-Lions theorem, we may prove the relation \eqref{Com-relation} directly. The proof of the convergence follows from the  following estimate

\begin{equation}\label{J}
 \int_{Q_T} (\partial_t H_{\kappa}(n_{\kappa}))^2  + \sup_{t\in[0,T]}  \int_{Q} |\nabla H_{\kappa}(n_{\kappa})|^2 \leq C,   
\end{equation}
that is $H_{\kappa}$ is bounded function in $L^2 (0,T; W^{1,2}(Q))$. Before we introduce our complete results,  the following assumptions about the model's components are made throughout this paper: We assume that the model is equipped with non-negative and ill-prepared initial data in the sens that   
\begin{align}
P^0_\kappa=p(n_\kappa^0) \in L^\infty(Q),  \ \ n_{\kappa} \in L^1(Q), \ \  \|n_\kappa^0 -n_0^0 \|_{L^1(Q)}\to 0 \quad as \ \ \kappa \to 0. \label{Cond1} \\
c_{\kappa} \in L^1(Q), \ \  \|c_\kappa^0 -c_0^0 \|_{L^1(Q)}\to 0  \quad as  \ \ \kappa \to 0. \label{Cond2}
\end{align}
 \textbf{Main results.} Before stating our main result, we first introduce the definition of the weak solution for the initial-boundary value problem of parabolic system.
\begin{defi} (Weak solution) We call that $(n_l,n_d,c)$ a solution of system \eqref{Eq1}-\eqref{Eq2} on $(0, T )$, if the components $(n_l, n_d, c)$ are all non-negative, and verify
\begin{itemize}
\item (H 1) $n_l, n_d, c \in L^\infty({Q}_T)$,
 \item (H 2) $ p, c \in L^{2}((0, T) ; W^{1,2}(Q)),$ 
\end{itemize}
\end{defi}
and  satisfy \eqref{Eq1}-\eqref{Eq2} in distributional sense. More precisely,  

\begin{equation}
\int_{Q}(n_l \varphi)|_{\xi=t}-\int_{Q} n_{l}^0 \varphi|_{\xi=0}+\int_{Q_{t}}(n_l \nabla p \nabla \varphi) dx dt=\int_{Q_{t}} R_1 \varphi dx dt,
\end{equation}
for all $t \in[0, T]$ and $\varphi \in C^{2}(\overline{Q_{T}})$ such that $\varphi \geq 0$ in $Q_{T}$ and $\varphi|_{\partial Q}=0$. Similarly, for $n_d$, and $c$ we have

\begin{align}
&\int_{Q}(n_d \varphi)|_{\xi=t}-\int_{Q} n_{d}^0 \varphi|_{\xi=0}+\int_{Q_{t}}(n_d \nabla p \nabla \varphi) dx dt=\int_{Q_{t}} R_2 \varphi dx dt, \\
&\int_{Q}(c \varphi)|_{\xi=t}-\int_{Q} c^0 \varphi|_{\xi=0}-d\int_{Q_{t}}(c \partial_{\xi} \varphi+c \Delta \varphi) dx dt +d\int_{0}^{t} \int_{\partial Q} c_{\infty} \partial_{\nu} \varphi = \int_{Q_{t}} F \varphi dx dt,
\end{align}
where, $R_1 \equiv G(c) n_l-K_{D}(c) n_l$, $R_2 \equiv K_{D}(c) n_l-\mu n_d $, $F= f(c)n_{l},$ and $\partial_{\nu}$ is the outward unit normal derivative at the boundary. Note that since we are looking for solutions $(n_l, n_d, c)$ with
 $$(n_l, n_d, c) \in (W^{1,1}((0,T),  W^{-1,1}(Q)))^3,$$
 we have 
  $$(n_l, n_d, c) \in (C((0,T),  W^{-1,1}(Q)))^3.$$
 
\begin{theo}\label{Theo-existance}
(Existence Theorem ) Assume the initial data are non-negative and smooth, namely, that 
\begin{itemize}
\item [(H1)] $c^{0} \in L^{\infty}(Q) \cap H^{1}(Q)$  
\item[(H2)] $n^{0}\in L^{\infty}(Q), \quad n_{d}^0, n_{l}^0 \in L^{\infty}(Q)\cap W_{0}^{1,2}(Q)$, 
\item[(H3)] $n_{l}^0, n_{d}^0, c^0 \geq 0, \quad \left\|n^{0}\right\|_{L^{\infty}(Q)} =1-\theta <1, \ \  \text{for some} \ \ \theta\in (0,1).$
\end{itemize}
Then there is a weak solution to \eqref{Eq1}-\eqref{Eq2}.
We set
\begin{equation}\label{unif}
c_{\max }:=\max \{|c_{0}|_{L^{\infty}(Q)}, |c_{\infty}|_{L^\infty ( \partial Q)} \}, \quad G_{m}= \sup_{\xi \in [0,c_{\max}]} G(\xi), \quad F_m= \sup_{\xi \in [0, c_{\max}]}f(\xi).
\end{equation}
\end{theo}
We recall that the global existence result was considered in \cite{Berr} for systems similar to  \eqref{Eq1}-\eqref{Eq2}, but with different reaction terms and boundary conditions. Here we take another approach to drive the existence results.
\begin{theo}\label{Theo-inco-limit}(Incompressible Limit)
Let the assumptions of Theorem \ref{Theo-existance} hold. Assume:
\begin{itemize}
    \item [(H 4)] $G'(\xi) < -\alpha , \quad \alpha >0, \quad   \text{and}  \ \ \xi \in [0, c_{\max}].$
    \item [(H 5)] $c_{\infty} =1,\ \ in \ \ \sum_{T}.$
    \item [(H 6)] $n^0_{\kappa}$ is measurable and, $ \int_{Q}\int_{0}^{n^{0}_{\kappa}} H_{\kappa}(\xi)d\xi <\infty .$
    \item [(H 7)] $\partial Q$ is a $C^{1,1}.$
\end{itemize}
Let's denote by $(n_\kappa, n_{l,\kappa}, n_{d,\kappa}, c_\kappa)$ the solution obtained in Theorem \eqref{Theo-existance}. Then as $\kappa \to 0$, we have the following convergences
\begin{gather}
 (n_\kappa, n_{l,\kappa}, n_{d,\kappa}, c_\kappa) \rightharpoonup (n_0, n_{l,0}, n_{d,0}, c_0) \quad w^{*}-\left(L^{\infty}\left(Q_{T}\right)\right)^{4}, \\
  (n_\kappa, n_{l,\kappa}, n_{d,\kappa}, c_\kappa) \rightarrow (n_0, n_{l,0}, n_{d,0}, c_0)  \quad  \text{strongly in}  \ \ (C(0, T; L^2(Q)))^{4}, \label{EQ9} \\
 H_{\kappa} \rightarrow p_{0}  \qquad \text{strongly in} \ \ C(0, T; L^2(Q)), \label{EQ10} \\
\nabla p_{\kappa} \rightharpoonup \nabla p_{0} \quad \text{weakly in} \ \  L^{2}(0, T ; W^{1,2}(Q)), \label{EQ11} \\
\nabla p_{\kappa} \rightarrow \nabla p_{0}  \qquad \text{strongly in}  \ \  L^{2}(0, T ;(L^{2}(Q)))^{d} \label{EQ12}, \\
 \end{gather}
 $(n_0, n_{l,0}, n_{d,0}, c_0)$ is the limit solution, which satisfies
 \begin{align}
     -\int_{Q_T} n_0 \partial_t \varphi dx dt + \int_{Q_T} \nabla p_0 \nabla \varphi dt dx  = \int_{Q_T} R_0 \varphi dx dt \\
     -\int_{Q_T} n_{l,0} \partial_t \varphi dx dt + \int_{Q_T} n_{l,0}\nabla p_0 \nabla \varphi dt dx  = \int_{Q_T} R_{1,0} \varphi  dx dt \\ 
       - \int_{Q_T} n_{d,0} \partial_t \varphi dx dt + \int_{Q_T} n_{l,0} \nabla p_0 \nabla \varphi dx dt=
        \int_{Q_T} R_{2,0} \varphi  dx dt \\ 
         - \int_{Q_T} c_{0} \partial_t \varphi dx dt + d\int_{Q_T}  \nabla c_0 \nabla \varphi dt dx  = -\int_{Q_T} f(c_0) n_{l,0} \varphi  dx dt \\ 
 \end{align}
 
\end{theo}
\begin{theo} (Convergence rate) \label{Theo-conv-rate}
 Under the assumptions of  Theorem \ref{Theo-inco-limit}, and the assumptions \eqref{Cond1}-\eqref{Cond2}. Then for all $T>0$ there exists a unique pair function $(n_{0},c_{0}) \in C^{1}(0,T; L^1(Q))$ such that   $(n_{\kappa}, c_{\kappa})$ converges strongly to $(n_{0}, c_{0})$ in $L^{\infty}(0,T; L^1(Q))$ with the following convergence rate 
\begin{equation}\label{Conv-n}
\sup_{t\in (0,T)}\|n_{\kappa}(t)-n_{0}(t)\|_{L^1(Q)} \leq \|n_{\kappa}(0)-n_{0}(0)\|_{L^1(Q)} + G_mt \|n_{\kappa}\|_{L^1(Q)}.
\end{equation} 
\begin{equation}\label{Conv-c}
\sup_{t\in (0,T)}\|c_{\kappa}(t)-c_{0}(t)\|_{L^1(Q)} \leq \|c_{\kappa}(0)-c_{0}(0)\|_{L^1(Q)} + F_m t \|n_{\kappa}\|_{L^1(Q)}.
\end{equation} 

\end{theo}

\subsection{A family of approximate problems and proof of the Theorem \ref{Theo-existance}} 
 We approximate the degenerate problem \eqref{Eq1}-\eqref{Eq2} by a sequence of non-degenerate parabolic problem,  and show that their solutions converge to the solution of the degenerate problem when the regularization parameter  
tends to zero. The regularization is adapted from \cite{Amman}. For small $\epsilon >0$ we define :

\begin{equation}
\chi_{\epsilon}(n):= \begin{cases}
\kappa \epsilon & \text { if } \quad n<0 \\
 \kappa \frac{(n+\epsilon)}{(1-n)^2} & \text { if } \quad 0 \leq n \leq 1-\epsilon \\
  \kappa\frac{1}{\epsilon^{2}} & \text { if } \quad n \geq 1-\epsilon
\end{cases}
\qquad \quad
\chi_{z,\epsilon}(n):= \begin{cases}
\kappa \epsilon & \text { if } \quad n<0 \\
 \kappa \frac{(z+\epsilon)}{(1-n)^2} & \text { if } \quad 0 \leq n \leq 1-\epsilon \\
  \kappa\frac{1}{\epsilon^{2}} & \text { if } \quad n \geq 1-\epsilon
\end{cases}
\end{equation}
thus the approximation of the system \eqref{Eq1}-\eqref{Eq2}  is given by
\begin{align}
\partial_t n_{\epsilon} &= div( \chi_\epsilon(n_{\epsilon}) \nabla n_{\epsilon}) +G(c_{\epsilon}) |n_{l, \epsilon}|- \mu |n_{d, \epsilon}|\equiv R_{\epsilon}, \qquad \text { in } Q_{T} \equiv Q \times(0, T) \label{Eq01} \\
\partial_t n_{l, \epsilon } &= div( \chi_{n_{l,\epsilon}}(n_{\epsilon})\nabla n_{\epsilon}) +G(c_{\epsilon}) |n_{l, \epsilon}|-K_{D}(c_{\epsilon}) |n_{l, \epsilon}|, \qquad \text { in } Q_{T} \equiv Q \times(0, T) \label{Eqa}  \\
\partial_t n_{d, \epsilon} &= div( \chi_{n_{d,\epsilon}}(n_{\epsilon})\nabla n_{\epsilon}) +K_{D}(c) |n_{l, \epsilon}|-\mu |n_{d, \epsilon}|, \qquad \text { in } Q_{T} \equiv Q \times(0, T) \label{Eqaa}\\
\partial_t c_{\epsilon} &=d \Delta c_{\epsilon} -f(c_{\epsilon}) n_{l,\epsilon}, \qquad \text { in } Q_{T} \label{Eqaaa} \\
\chi_{\epsilon}(n_{\epsilon}) \partial_\nu n_{\epsilon}&= \chi_{n_{l,\epsilon}}(n_\epsilon) \partial_\nu n_{l, \epsilon} =  \chi_{n_{d,\epsilon}}(n_{\epsilon})\partial_\nu n_{ d,\epsilon} =0 \qquad \text { on } \Sigma_{T} \equiv \partial Q \times(0, T), \\
c_\epsilon &=c_{\infty} \text { on } \Sigma_{T}, \\
(n_l|_{t=0}, n_d|_{t=0}, c|_{t=0})&=(n_{l, \epsilon}^0, n_{d, \epsilon}^0, c_{\epsilon}^0), \label{Eq02}
\end{align}
which is non-degenerate parabolic. Therefore, the problem \eqref{Eq01}-\eqref{Eq02} possesses a solution, which we denote by ($n_\epsilon, n_{l,\epsilon}, n_{d,\epsilon}, c_\epsilon$).  
\begin{proof} 
We first show the non-negativity and boundedness of all components of the solution to the non-degenerate approximation using the comparison principle (see, \cite{Amman}).
Per hypothesis, the boundary/initial data for our system are non-negative. Moreover, we see that 
$$
R_1(0,n_{d, \epsilon},c_{\epsilon})=f(0)=0,
$$ 
 thus, zero is subsolution to $n_{l, \epsilon}$ and $c_{\epsilon}$. According to (A3), and since $n_{l, \epsilon}$ and $c_{\epsilon}$ are non-negative, we have

$$
R_2(n_{l, \epsilon},0,c_{\epsilon})=K_D(c_\epsilon)n_{l, \epsilon} \geq 0, 
$$
we conclude the non-negativity of $n_{d, \epsilon}$. The non-negativity of $n_\epsilon$ is an immediate consequence.
Now we will establish the $L^\infty$ a priori estimates for $n_\epsilon$ using the so-called true upper-barriers technique; we set

$$n_\epsilon= n_{l, \epsilon}+ n_{d, \epsilon},$$
we introduce the barrier function $n_\omega  =\omega +1$, where $\omega$ is a solution of the  following elliptic problem  
$$
\begin{aligned}
\Delta \omega &=-1 \quad \text { in } Q, \\
\omega|_{\partial Q} &=0 .
\end{aligned}
$$
The maximum principle for elliptic equations \cite{Ren} implies that  $\omega \geq 0$ in $Q$ and 

$$1 \leq n_{\omega}(x) \leq 1+K, \quad x \in Q,$$
for some constant $ K \geq 0$. Furthermore, we observe that $ n^0_{\epsilon}=n_{\epsilon}|_{t=0} \leq n_{\omega}|_{t=0}, \quad 0=n_{\epsilon}|_{\partial Q} \leq n_{\omega}|_{ \partial Q}.$ 
Since $n_{\omega}$ is a time-independent function, it follows that
$$
\begin{aligned}
&\partial_{t} n_{\omega}-div(\chi_{\epsilon}(n_{\omega}) \nabla n_{\omega})- G(c_\epsilon)n_{l, \epsilon}+\mu n_{d, \epsilon}\\
&=0+\frac{\kappa}{\epsilon^{2}}- G(c_\epsilon)n_{\omega}\geq \frac{\kappa}{\epsilon^{2}} - G_m(1+K)\geq 0 \\ 
&= \partial_{t} n_{\epsilon}-div(\chi_{\epsilon}(n_{\epsilon}) \nabla n_{\epsilon})- G(c_\epsilon)n_{l, \epsilon}+\mu n_{d, \epsilon}\\
\end{aligned}
$$
for all sufficiently small $\epsilon >0$. Hence $n_{\omega}$ is an upper-solution for the total density $n_\epsilon$ i.e $n_{\omega} \geq n_{\epsilon}$. It follows that,  by the parabolic comparison principle there exists an $\epsilon_0'$ such that 
$$0 \leq n_\epsilon \leq C^*_{\epsilon}$$ 
for all $0<\epsilon < \epsilon_0'$, and  $C^*_{\epsilon}$  is a positive constant, which also implies
$$0 \leq n_{l, \epsilon},  n_{d, \epsilon}  \leq C^*_{\epsilon}.$$
To show the uniform boundedness of $c_{\epsilon}$, we define the constant 
$c_{\max }:=\max \{|c^0|_{L^{\infty}(Q)}, |c_{\infty}|_{L^\infty ( \partial Q)} \}$ 

$$\partial_{t} c_{\max }- d\Delta c_{\max }+ f(c_{\max}) n_{ l, \epsilon} = f(c_{\max}) n_{l, \epsilon} \geq 0,$$
where we used the nonnegativity of $n_{l, \epsilon}$, this shows that if $c_{\max}$ is sufficiently large, it is an upper solutions for $c_{\epsilon}$, and, thus, the uniform boundedness using the comparison principle.
\end{proof}
The following lemma shows that the total population $n$ is uniformly bounded away from the singularity, i.e., the case $n=1$ is never attained for all sufficiently small $\epsilon > 0$.
\begin{lemma}\label{nbound}
 Under the hypothesis of Theorem \eqref{Theo-existance}, there exist $\rho>0$ and $\epsilon_{0}>0$ such that $n_{\epsilon} \leq 1-\rho$ in $Q_{T}$ for all $0<\epsilon<\epsilon_{0}$.
\end{lemma}
\begin{proof}
Let $\epsilon_{0}>0$, we construct a suitable barrier function and consider the elliptic problem
\begin{equation}\label{Eq3}
\begin{aligned}
\begin{cases}
\Delta \varphi(x) &=-\lambda_{1}, \quad  x \in Q, \\
\varphi(x)_{|_{x \in \partial Q}} &=\lambda_{2}.
\end{cases}
\end{aligned}
\end{equation}
The constants $\lambda_{1}$ and $\lambda_{2}$ are defined by
$$
\begin{aligned}
&\lambda_{1}:=\sup_{ 0<\epsilon <\epsilon_0}\left\{ \|R_{\epsilon}\|_{L^{\infty}(Q_{T})}\right\}, \\
&\lambda_{2}:= \sup_{ 0<\epsilon <\epsilon_0}\left\{ \|H_{\epsilon}(n^0)\|_{L^{\infty}(Q)}\right\},
\end{aligned}
$$
where $H_{\epsilon}$ is given by 
\begin{equation}\label{H-epsilon}
H_{\epsilon}(n^0)= \int_{0}^{n^0}\chi_{\epsilon}(\xi)d\xi= \int_0^{n^0}\kappa\frac{\xi+\epsilon}{(1-\xi)^2}d\xi.
\end{equation}
We point out that for sufficiently small $\epsilon_1 <\epsilon_0'$, the constants $\lambda_1$ and $\lambda_2$ can be chosen uniform for all $0<\epsilon_1 <\epsilon_0$. Moreover, 
the solution $\varphi$ of \eqref{Eq3} is bounded on $Q$, and
by the maximum principle for elliptic problems  \cite{Ren} it follows that $\varphi>\lambda_2$ in $Q$.
For $\epsilon<\epsilon_{1}$ , we define $ \Upsilon_{\epsilon}:=H_{\epsilon}^{-1}(\varphi)$ and observe that 

$$\partial_{t} \Upsilon_{\epsilon}-\Delta\left(H_{\epsilon}\left(\Upsilon_{\epsilon}\right)\right)=\lambda_{1}=\left\|R_\epsilon\right\|_{L^{\infty}\left(Q_{T}\right)} \geq \partial_{t} n_{\epsilon}-\Delta\left(H_{\epsilon}\left(n_{\epsilon}\right)\right),$$
in $Q_{T}$. Furthermore, the boundary conditions indicate that
$$
\Upsilon_{\epsilon}|_{\partial Q}=H_{\epsilon}^{-1}(\varphi)|_{\partial Q}=H_{\epsilon}^{-1}(\lambda_{2}) \geq  n_{\epsilon}|_{\partial Q}=0.
$$
Since the function $H_{\epsilon}^{-1}$ is a monotone function, the initial data then satisfies
$$
\Upsilon_{\epsilon}|_{t=0}=H_{\epsilon}^{-1}(\varphi)|_{t=0} \geq H_{\epsilon}^{-1}(\lambda_{2}) \geq H_{\epsilon}^{-1}(H_{\epsilon}(n^{0}))=n^{0}.
$$
As a result, the function $\Upsilon_{\epsilon}$ is an upper solution for $n_{\epsilon}$. Using the fact that $\varphi$ is bounded in $Q$ and that $H_{\epsilon}$ converges point-wise to infinity in the interval $(0,1)$, we conclude that there exist $0<\epsilon_{0} \leq \epsilon_{1}$ and $\rho \in(0,1)$ such that $n_{\epsilon} \leq P_{\epsilon}=H_{\epsilon}^{-1}(\varphi)<1-\rho$ for all $0<\epsilon<\epsilon_{0}$. 
We particularly emphasize that the nondegenerate approximations for the total population are uniformly bounded away from the singularity i.e
$$\left\|n_{\epsilon}(\cdot, t)\right\|_{L^{\infty}(Q)} \leq 1-\rho, \quad \forall \ \ \epsilon<\epsilon_{0}, \ \ t \geq 0.$$
\end{proof}
Next we show that the solutions of the regularized system \eqref{Eq01}-\eqref{Eq02} converge as $\epsilon \to 0$ to a solution of the degenerate problem $(n, n_{l}, n_{d}, c)$. On the other hand we have $\lim_{\epsilon \to \infty} H_{\epsilon}(1)= \infty$, since $n_{\epsilon}$ is uniformly bounded we can find $\beta_{\rho}$ such that $H_{\epsilon}(n_{\epsilon}(t,x))$ is uniformly bounded from above and below for all sufficiently small $\epsilon$ i.e:  
 $$0\leq H_{\epsilon}(n_{\epsilon}(t,x)) < \beta_{\rho} < \infty.$$ 
We write the equation \eqref{Eq01} in terms of the Laplacian
$$\partial_{t} n_{\epsilon}=\Delta H_{\epsilon}(n_{\epsilon})+R_{\epsilon}$$
where $R_{\epsilon} \in L^{\infty}(Q_T)$. For $T>0$, we define the local parabolic cylinder $Q_T: Q \times (T, T+1]$. From the classical theory of parabolic equations \cite{Lad}, it is known that the solution $n_{\epsilon} \in\mathcal{C}^{\alpha}(Q_T)$ for some $\alpha>0$ and
\begin{equation}\label{n-Holder}
\|n_{\epsilon}\|_{\mathcal{C}^{\alpha}(Q_{T})} \leq K(\| R_{\epsilon} \|_{L^{\infty}(Q_{T})}),
\end{equation}
\begin{equation}\label{c-Holder}
\|c_{\epsilon}\|_{\mathcal{C}^{\alpha}(Q_{T})} \leq W(\|F_{\epsilon}\|_{L^{\infty}(Q_{T})}),
\end{equation}
where $K$ and $W$  are non-decreasing functions that depend on the upper and lower bounds
of $H_{\epsilon}$ and $c_{\epsilon}$ respectively. Since $R_{\epsilon}$ is uniformly bounded (relatively to $\epsilon$), $n_{\epsilon}$ is bounded in $\mathcal{C}^\alpha(Q_T)$, which is compactly embedded in $\mathcal{C}(\overline{Q_T})$, thus we have a strong convergence in  $\mathcal{C}(Q_T)-$norm as $\epsilon \to 0$, i.e
 $$n_{\epsilon} \to n_*, \qquad  \text{in} \ \  \mathcal{C}(\overline{Q_T}).$$
Since $n_{l, \epsilon}$ and  $n_{d, \epsilon}$ are uniformly bounded under the $L^\infty(Q_T)$ norm, it follows when $\epsilon \to 0$ 
$$n_{l, \epsilon}\rightharpoonup n_{l,*} \qquad
n_{d, \epsilon}\rightharpoonup n_{d,*} \qquad\ w-* \ \ \text{in} \ \  L^\infty(Q_T).$$
We have $n_{\epsilon}= n_{l, \epsilon}+n_{d,\epsilon}$, and due to the uniqueness of the limit we get  $n_*=n_{l,*}+n_{d,*}.$
It remains to prove now that the functions $n_{l,*}$ and $n_{d,*}$ are indeed solutions of the original system. As a weak solution, $n_{l,*}$ has to satisfy for any $\varphi \in C^2(Q_T)$
\begin{equation}\label{Convv}
\int_{Q}(n_{l,*} \varphi)|_{\xi=T}-\int_{Q} (n_{l,*} \varphi)|_{\xi=\tau} + \int_{Q_{T}}(\chi_{n_{l,*}}(n_*) \nabla n_* \nabla \varphi)=\int_{Q_{t}} (G(c_*) n_{l,*}-K_{d}(c_*)n_{l,*}) \varphi
\end{equation}
where $0 \leq \tau \leq t \leq T$. We know that for any $\epsilon >0$ the following equation is verified for $n_{l,\epsilon}$
\begin{equation}\label{Conv}
\int_{Q}(n_{l,\epsilon} \varphi)|_{\xi=T}-\int_{Q} (n_{l,\epsilon} \varphi)|_{\xi=\tau}+\int_{Q_{T}}( \chi_{n_{l,\epsilon}}(n_\epsilon) \nabla n_{\epsilon}  \nabla \varphi)=\int_{Q_{T}} (G(c_\epsilon) n_{l,\epsilon}-K_{d}(c_\epsilon)n_{l,\epsilon}) \varphi
\end{equation}
 As $\epsilon \to 0$ in \eqref{Conv} the convergence of the third term is not clear, thus we give estimates for the residual term 

\begin{equation}
\Upsilon_{\epsilon}= \int_{Q_{T}}( \chi_{n_{l,*}}(n_*) \nabla n_{*}  \nabla \varphi) -  \int_{Q_{T}}( \chi_{n_{l,\epsilon}}(n_\epsilon) \nabla n_{\epsilon}  \nabla \varphi)
\end{equation} 
and show that $\Upsilon_{\epsilon}$ vanishes as $\epsilon \to 0$. The idea is we treat the region, where the total density is small enough.  For $\delta \in (0,1)$ and sufficiently small $\epsilon_0 >0$ we define the open set 
$$Q_{T,\delta}=\{(x,t) \in Q_T : \ \ (n_{l,\epsilon}+n_{d, \epsilon}) <\delta,   \quad \text{and} \quad (n_{l,*}+n_{d, *}) <\delta,  \quad \forall \epsilon < \epsilon_0  \},$$ 
and decompose $\Upsilon_{\epsilon}$ over  $Q_{T,\delta}$ and its complement $Q_{T,\delta}^c$. To this end, we define the integrals 
\begin{align*}
\Upsilon_{\epsilon}&= I_{\epsilon}(\delta) +J_{\epsilon}(\delta), \\
&:= \int_{Q_{T, \delta}}( \chi_{n_{l,*}}(n_*) \nabla n_{*}  \nabla \varphi) -  \int_{Q_{T, \delta}}( \chi_{n_{l,\epsilon}}(n_\epsilon) \nabla n_{\epsilon}  \nabla \varphi), \\
&+ \int_{Q_{T, \delta}^c}( \chi_{n_{l,*}}(n_*) \nabla n_{*}  \nabla \varphi) -  \int_{Q_{T, \delta}^c}( \chi_{n_{l,\epsilon}}(n_\epsilon) \nabla n_{\epsilon}  \nabla \varphi), 
\end{align*}
\textbf{ $\bullet$ Istimates for $I_{\epsilon}(\delta)$}: 
Multiplying the equation  \eqref{Eq01}  by $n_{\epsilon}$ and integrating over $Q$, we obtain

$$
\frac{1}{2} \frac{d}{dt}\|n_{\epsilon}\|_{L^2(Q)}^2+ \int_{Q} \chi_{\epsilon}(n_{\epsilon})|\nabla n_{\epsilon}|^2dx= \int_{Q} (G(c_{\epsilon})n_{l, \epsilon}- \mu n_{d,\epsilon})n_{\epsilon},
$$
integrating  over time we get 
$$
\frac{1}{2} \|n_{ \epsilon}(T)\|_{L^2(Q)}^2+  \int_{Q_T} \chi_{\epsilon}(n_{\epsilon})|\nabla n_{\epsilon}|^2dt dx= \int_{Q_T} (G(c_{\epsilon}) n_{l, \epsilon} - \mu n_{d, \epsilon})n_{\epsilon} dt dx + \frac{1}{2} \|n_{ \epsilon}^0\|_{L^2(Q)}^2.
$$
From the uniform boundness of the cell density, we deduce
$$
 \int_{Q_T} \chi_{n_{l,\epsilon}}(n_{\epsilon})|\nabla n_{\epsilon}|^2dt dx \leq   \int_{Q_T}  \chi_{\epsilon}(n_{\epsilon})|\nabla n_{\epsilon}|^2dt dx \leq  \int_{Q_T} (G(c_{\epsilon}) n_{l, \epsilon} - \mu n_{d, \epsilon})n_{\epsilon} dt dx + \frac{1}{2} \|n_{\epsilon}^0\|_{L^2(Q)}^2 <C,
$$
for some constant $C>0$. From Lemma \ref{nbound}, the constant $C$ is independent of $\epsilon >0$.
Moreover, we can write 
$$
\chi_{n_{l,\epsilon}}(n_{\epsilon})|\nabla n_{\epsilon}|^2= \sqrt{ \chi_{n_{l,\epsilon}}(n_{\epsilon})} \sqrt{\chi_{n_{l,\epsilon}}(n_{\epsilon})} |\nabla n_{\epsilon}|^2,
$$
 thus we have shown
\begin{equation}\label{Q-T}
 \left \| \sqrt{\chi_{n_{l,\epsilon}}(n_{\epsilon})}\nabla n_{\epsilon}\right \|_{L^2(Q_T)} \leq C,
\end{equation} 
this allows us to use H\"older’s inequality to estimate the second integral of $I_{\epsilon}$ 
\begin{align*}
\left | \int_{Q_{T}}\chi_{n_{l,\epsilon}}(n_{\epsilon}) \nabla n_{\epsilon} \nabla \varphi dx dt \right| &\leq \left \|\sqrt{\chi_{n_{l,\epsilon}}(n_{\epsilon})}\nabla n_{\epsilon} \right \|_{L^2(Q_T)} \left \| \sqrt{\chi_{n_{l,\epsilon}}(n_{\epsilon})} \nabla \varphi \right \|_{L^2(Q_{\delta,T})}\\
&\leq C \left( \int_{Q_{\delta, T}}  \kappa\frac{n_{l, \epsilon}+\epsilon}{(1-n_{\epsilon})^2} |\nabla \varphi|^2\right)^{\frac{1}{2}} \leq C \frac{\sqrt{3\delta}}{(1-2\delta)}\|\varphi\|_{L^2((0,T), H^1(Q))}.
\end{align*} 
 We follow the same steps to estimate the second term of the integral $I_{\epsilon}(\delta)$, therefore the existence of $\epsilon_0 >0$ such that for all sufficiently small $\epsilon <\epsilon_0$
\begin{equation}\label{I-est}
I_{\epsilon}(\delta) \leq C \|\varphi\|_{L^2((0,T),H^1(Q))}.
\end{equation}
Next we estimate $J_{\epsilon}(\delta)$ for $\delta >0$.  Restricted to the domain $Q_{T, \delta}^c= Q_T \backslash \overline{Q_{T, \delta}}$. The solution $n_{\epsilon}$ satisfy the estimate of Lemma \ref{nbound} and \eqref{n-Holder} uniformly,  and the H\"older exponent $\alpha$ is independent of $\epsilon >0$. 
Indeed, if $\epsilon > 0$ is sufficiently small, then $n_{\epsilon} >\delta$ in the region $Q_{T, \delta}^c$. Consequently, the term 
\begin{equation}\label{Q-Tc}
\kappa\frac{\delta}{2} \leq (\kappa\frac{\delta}{2}+\epsilon)\leq \chi_{ n_{l,\epsilon}}(n_{\epsilon}) \leq \frac{1}{\eta^2} \qquad (x, t)\in Q_{T, \delta}^c.
\end{equation}
is uniformly bounded from above and below by a positive constant which is independent of $\epsilon > 0$.
Now  we need to show that
\begin{equation}
\nabla n_{l,\epsilon} \rightharpoonup \nabla n_{l,*} \qquad L^2(Q_{T, \delta}^c).  
\end{equation}
From \eqref{Q-Tc}, and \eqref{Q-T} we have
$$
\int_{Q_{T, \delta}^c} |\nabla n_{l, \epsilon}|^2 \leq \frac{C}{\delta},
$$
where $C$ is a constant independent of $\epsilon$. This implies weak-convergence for a subsequence in $L^2(Q_{T,\delta}^c)$. Moreover, we have  $\nabla n_{l, \epsilon}$
converges in $\mathcal{D'}(Q_{T, \delta}^c)$ to $\nabla n_{l, *}$, because of the uniform convergence in $\mathcal{C}(Q_{T, \delta})$, the limit in $\mathcal{D'}(Q_{T,\delta}^c)$ is unique and, thus, we obtain the weak convergence of $\nabla n_{l, \epsilon}$ to $\nabla n_{l,*}$ over $Q_{T, \delta}^c$.
To resume, from Lemma \ref{nbound} we have as $\epsilon \to 0$  
$$\chi_{n_{l,\epsilon}}{(n_{\epsilon})} \longrightarrow \chi_{n_{l,*}}(n_{*}) \qquad in \ \ \mathcal{C}(Q_T),$$
and from the weak convergence
of  $\nabla n_{l, \epsilon}$ to $\nabla n_{l, *}$, we obtain finally for every $\delta >0$,
$$
\lim_{\epsilon \to 0} J_{\epsilon} (\delta)=0.
$$

To conclude the proof, we pick any $\eta >0$. From \eqref{I-est}
we obtain that there exists $\delta(\eta)$, and $\epsilon \leq \epsilon_0( \delta (\eta
))$, such that $I_{\epsilon}(\delta(\eta))\leq \frac{\eta}{2}$. Since $J_{\epsilon}(\delta) \to 0$ there exist $\epsilon_1(\eta)$ such that $J_{\epsilon}(\delta(\eta))\leq \frac{\eta}{2}$ for small enough $\epsilon < \epsilon_1(\delta (\eta))$. $\Upsilon_{\epsilon}= I_{\epsilon}+J_{\epsilon}$ is independent of $\delta$.

$\forall \eta>0, \quad \exists \epsilon(\eta)= \min \{\epsilon_0(\eta), \epsilon_1(\eta)\}$, such that for $\epsilon \leq \epsilon(\eta)$, we have $R_{\epsilon} \leq \eta$. Thus $R_{\epsilon}\to 0$  as $\epsilon \to 0$. The same procedure can be carried out for the remaining
components.


\section{Incompressible limit as $\kappa \to 0$ (Proof of Theorem \ref{Theo-inco-limit})}
This section is devoted to the proof of the incompressible limit, i.e when $\kappa \to 0$. Thanks to the result proven in the previous section, cf. Theorem \ref{Theo-existance}  we know that for each $\kappa >0 $ there exists  $(n_{\kappa},n_{l,\kappa},n_{d, \kappa},c_{\kappa})$  that verify \eqref{Eq1}-\eqref{Eq2} in a weak sens 
\begin{align}
\partial_{t} n_{\kappa}- \Delta (H_{\kappa}(n_{\kappa}) &=G(c_\kappa) n_{l,\kappa}- \mu n_{d,\kappa} \equiv R_{\kappa} \quad \text { in } \mathcal{D'}(Q_{T}), \label{Q1} \\
\partial_t n_{l,\kappa}- div(n_{l,\kappa} \nabla p_\kappa) &=G(c_\kappa) n_{l,\kappa} -K_{D}(c_\kappa) n_{l,\kappa}, \qquad \text { in } \mathcal{D'}(Q_{T}) \label{Q2} \\
\partial_t n_{d,\kappa} - div( n_{d,\kappa}\nabla p_\kappa) &=K_{D}(c_\kappa) n_{l,\kappa}-\mu n_{d,\kappa}, \qquad \text { in }  \mathcal{D'}(Q_{T})\label{Q3}\\
\partial_t c_\kappa &=d \Delta c_\kappa -f(c_\kappa) n_{l,\kappa} , \qquad \text { in } \mathcal{D'}(Q_{T})  \label{Q4}\\
n_{l,\kappa} \nabla p_\kappa \cdot \mathbf{n} &=n_{d,\kappa} \nabla p_\kappa \cdot \mathbf{n}=0 \qquad \text { on } \Sigma_{T} \equiv \partial Q \times(0, T), \label{Q5}\\
c &=c_{\infty} \text { on } \Sigma_{T}, \label{Q6} \\
\left(n_{\kappa}, n_{\kappa}, n_{\kappa}, c_{\kappa}\right)\big|_{t=0}  &=\left(n_{\kappa}(x,0)+\kappa ,n_{l,\kappa}(x, 0)+\kappa,  n_{d,\kappa}(x, 0)+\kappa, c_{\kappa}(x,0)+\kappa\right) \text { on } Q \label{Q7}.
\end{align}
From the previous section, we have 
\begin{equation}\label{Q8}
n_{1, \kappa}\geq 0, \ \ n_{d,\kappa} \geq 0, \ \ c_{\kappa} \geq 0, \ \ n_{\kappa} \leq n_{M}, \ \ c_{\kappa} \leq c_{\max},
\end{equation} 
where $n_{M}$ is independent of $\kappa$. Thus as $\kappa \to 0,$ we have 
\begin{equation}
(n_{\kappa} , n_{d,\kappa}, n_{l,\kappa}, c_{\kappa} ) \rightharpoonup (n_{0} , n_{d,0}, n_{l,0}, c_{0} ) ,\qquad w^*-L^{\infty}(Q_T).
\end{equation}
which also implies 
\begin{equation}
R_{\kappa} \rightharpoonup R_{0} =G(c_{0})n_{l, 0}-\mu n_{d, 0}, \qquad w^*-L^{\infty}(Q_T).
\end{equation} 
\begin{lemma}\label{Lemma1}
Let the assumptions of Theorem \ref{Theo-existance} and (A1)-(A4) hold, the following estimates hold true for all $T>0$, with constants $C=C(T)>0$
\begin{align}
\int_{Q_{T}}(\partial_{t} c_{\kappa})^{2} &+ \sup _{t \in[0, T]} \int_{Q}|\nabla c_{\kappa}|^{2}\leq C. \label{EqC}\\
 \int_{Q_{T}} (\partial_{t} n_{l,\kappa})^{2}& + \sup_{t \in[0, T]} \int_{Q}|\nabla n_{l,\kappa}|^2 \leq C. \label{EqL}
 \\
\int_{Q_{T}} (\partial_{t} n_{d,\kappa})^{2}& + _{t \in[0, T]} \int_{Q} |\nabla n_{d,\kappa}|^2 \leq C. \label{EqD}
\end{align}
\end{lemma} 
\begin{proof}
We multiply the  equation \eqref{Q4} by  $\partial_{\xi} c_{\kappa}$ and integrate in space and time, 
$$
\int_{\tau}^{t} \int_{\Omega} (\partial_{\xi} c_{\kappa})^{2}=-\frac{d}{2} \int_{\tau}^{t} \int_{\Omega} \partial_{\xi}|\nabla c_{\kappa}|^{2}+\int_{\tau}^{t} \int_{\Omega} \partial_{\xi} c_{\kappa} f(c_{\kappa}) n_{\kappa, l},
$$
where $0 \leq \tau \leq t \leq T$. We apply Young's inequality on the second r.h.s 
$$
\begin{aligned}
&\int_{\tau}^{t} \int_{Q}|\partial_{\xi} c_{\kappa}|^{2}+\frac{d}{2} \int_{Q}|\nabla c_{\epsilon}|^{2}|_{\xi=t} \\
&=\frac{d}{2} \int_{Q}|\nabla c_{\kappa}|^{2}|_{\xi=\tau}+\int_{\tau}^{t} \int_{Q} \partial_{\xi} c_{\kappa} f(c_{\kappa}) n_{\kappa,l} \\
&\leq\frac{d}{2} \int_{Q}|\nabla c_{\kappa}|^{2}|_{\xi=\tau}+ \frac{1}{2}\int_{\tau}^{t} \int_{Q}|\partial_{\xi} c_{\kappa}|^{2}+ \frac{1}{2}\int_{\tau}^{t} \int_{Q} |f(c_{\kappa}) n_{\kappa,l}|^2,
\end{aligned}
$$
Letting $\tau$ be zero, we get 
$$
\begin{aligned}
& \frac{1}{2} \int_{Q_{t}}|\partial_{\xi} c_{\kappa}|^{2}+\frac{1}{2} \int_{Q}|\nabla c_{\kappa}|^{2}|_{\xi=t} \\
&\leq \frac{1}{2} \int_{Q}|\nabla c^{0}|^{2}+\frac{1}{2}\int_{Q_{t}}|f(c_{\kappa})n_{\kappa, l}|^{2},
\end{aligned}
$$
\eqref{unif} and \eqref{Q8}  implies that
\begin{equation}\label{EQ3}
\int_{Q_{T}}\left|\partial_{t} c_{\kappa}\right|^{2} \leq C^{\prime}, \quad \sup _{t \in[0, T]} \int_{Q}\left|\nabla c_{\kappa}\right|^{2} \leq C'.
\end{equation}
To derive the estimates for  $n_{l,\kappa}$, we multiply equation \eqref{Q2} by $\partial_{\xi}n_{l, \kappa}$ and integrate in space and time we get
\begin{align*}
\int_{\tau}^{t} \int_{Q}|\partial_{\xi} n_{l,\kappa}|^{2}=- \int_{\tau}^{t} \int_{Q} n_{l, \kappa} \nabla p_{\kappa}\partial_{\xi}(\nabla n_{l,\kappa})+\int_{\tau}^{t} \int_{Q} \partial_{\xi} n_{l,\kappa} \left(G(c_{\kappa})n_{l, \kappa}- K_D(c_{\kappa})n_{l, \kappa}\right).
\end{align*}
Young's inequality implies
\begin{align*}
\int_{\tau}^{t} \int_{Q}|\partial_{\xi} n_{l,\kappa}|^{2} + \frac{1}{2} \int_{Q} |\nabla n_{l,\kappa}|^2|_{\xi=t} & \leq -\frac{1}{2} \int_{\tau}^{t} \int_{Q} |n_{l, \kappa} \nabla p_{\kappa}|^2 +  \frac{1}{2}\int_{Q}  |\nabla n_{l,\kappa}|^2|_{\xi=\tau} \\
&+ \frac{1}{2}\int_{\tau}^{t} \int_{Q} |\partial_{\xi} n_{l,\kappa}|^2 +  \frac{1}{2} \int_{\tau}^t \int_{Q} |\left(G(c_{\kappa})n_{l, \kappa}- K_D(c_{\kappa})n_{l, \kappa}\right)|^2,
\end{align*}
letting $\tau =0$, we obtain 

\begin{align*}
\frac{1}{2} \int_{Q_T}|\partial_{\xi} n_{l,\kappa}|^{2} + \frac{1}{2} \int_{Q} |\nabla n_{l,\kappa}|^2 & \leq  \frac{1}{2}\int_{Q}  |\nabla n_{l,\kappa}^0|^2 +\frac{1}{2} \int_{Q_T} |\left(G(c_{\kappa})n_{l, \kappa}- K_D(c_{\kappa})n_{l, \kappa}\right)|^2,
\end{align*}
the assumptions A2-A3 lead to 
\begin{equation}\label{EQ1}
\int_{Q_T}|\partial_{t} n_{l,\kappa}|^{2} \leq C', \qquad \qquad \sup_{t\in [0,T]}\int_{Q} |\nabla n_{l,\kappa}|^2 \leq C'.
\end{equation}
To drive estimates for $n_{d, \kappa}$, we follow the same procedures as above 
\begin{equation}\label{EQ2}
\int_{Q_T}|\partial_{t} n_{d,\kappa}|^{2} \leq C', \qquad \qquad \sup_{t\in [0,T]}\int_{Q} |\nabla n_{d,\kappa}|^2 \leq C'.
\end{equation}
\end{proof}
Our main a priori estimate is the following:
\begin{lemma}
Thanks to the estimates provided in Lemma \ref{Lemma1} , we have 
\begin{equation}\label{EqN}
\int_{Q_{T}} \left( \partial_{t} H_{\kappa}(n_{\kappa})\right)^2+\sup _{t \in[0, T]} \int_{Q} |\nabla H_{\kappa}(n_{\kappa})|^{2} \leq C  
\end{equation}
\end{lemma}
\begin{proof}
To derive the estimates for  $H_{\kappa}(n_{\kappa})$, let 
\begin{equation}
\Psi(n_{\kappa}, n_{\kappa, l}, n_{\kappa, d}, c_{\kappa}):=\int_{0}^{n_{\kappa}} \chi_{\kappa}(\zeta) \psi(\zeta, n_{\kappa, l}, n_{\kappa, d}, c_{\kappa}) d \zeta ,
\end{equation}
where $\chi_{\kappa}(\zeta)=\kappa\zeta(1-\zeta)^{-2}$, and $\psi(\zeta, n_{\kappa, l}, n_{\kappa, d}, c_{\kappa})=G(c_{\kappa})n_{l, \kappa}-\mu n_{d, \kappa} -\zeta +\zeta$ .\\
We have
 \begin{equation}\label{Ide}
\int_{\tau}^{t} \int_{Q} \partial_{\xi}  \Psi(n_{\kappa}, n_{l,\kappa}, n_{d,\kappa}, c_{\kappa})=\int_{Q} \Psi(n_{\kappa}, n_{l,\kappa}, n_{d,\kappa}, c_{\kappa})\bigg|_{\xi=t}- \int_{Q} \Psi(n_{\kappa}, n_{l,\kappa}, n_{d,\kappa}, c_{\kappa}) \bigg|_{\xi=\tau}.
\end{equation}
Multiplying the equation of the total density \eqref{Q1} by $\partial_{\xi}(H(n_{\kappa}))$ and integrating over space and time, we obtain
$$
\begin{aligned}
& \int_{\tau}^{t} \int_{Q} \chi_{\kappa} (n_{\kappa})(\partial_\xi n_{\kappa})^2
= \int_{\tau}^{t} \int_{Q} \partial_{\xi}(H(n_{\kappa})) \partial_{\xi} n_{\kappa}\\
&= \int_{\tau}^{t} \int_{Q} div(\chi_{\kappa}(n_{\kappa}) \nabla n_{\kappa}) \partial_{\xi}(H(n_{\kappa}))+\int_{\tau}^{t} \int_{Q} \partial_{\xi}(H(n_{\kappa})) \psi( n_{\kappa}, n_{\kappa, l}, n_{\kappa,d}, c_{\kappa})\\
&= -  \int_{\tau}^{t} \int_{Q} \chi_{\kappa}(n_{\kappa}) \nabla n_{\kappa} \partial_{\xi}(\nabla H(n_{\kappa}))+\int_{\tau}^{t} \int_{Q} \partial_{\xi}(H(n_{\kappa})) \psi( n_{\kappa}, n_{\kappa, l}, n_{\kappa,d}, c_{\kappa}) \\
&= -\frac{1}{2}\int_{\tau}^{t} \int_{Q} \partial_\xi |\chi_{\kappa}(n_{\kappa}) \nabla n_{\kappa}|^2 +\int_{\tau}^{t} \int_{Q} \partial_{\xi}(H(n_{\kappa})) \psi( n_{\kappa}, n_{\kappa, l}, n_{\kappa,d},c_{\kappa}). 
\end{aligned}
$$
Where we used integration by parts in the second step. Thus, using the
identity \eqref{Ide} we obtain
$$
\begin{aligned}
\int_{\tau}^{t} \int_{Q} \partial_{\xi} \Psi(n_{\kappa}, n_{l,\kappa}, n_{d,\kappa}, c_{\kappa})   &=\int_{\tau}^{t} \int_{Q} \int_{0}^{n_{\kappa}} \chi_{\kappa}(\zeta) \partial_{\xi}  \psi(\zeta, n_{l,\kappa}, n_{d,\kappa}, c_{\kappa})+\int_{\tau}^{t} \int_{Q} \partial_{\xi} n_{\kappa} \chi_{\kappa}(n_{\kappa}) \psi(n_{\kappa}, n_{l,\kappa}, n_{d,\kappa}, c_{\kappa}) \\
&=\int_{\tau}^{t} \int_{Q} \int_{0}^{n_{\kappa}} \chi_{\kappa}(\zeta)(\partial_{\xi} n_l \partial_{n_l} \psi(\zeta, n_{l,\kappa}, n_{d,\kappa}, c_{\kappa})+\partial_{\xi} n_{d} \partial_{n_d} \psi(\zeta, n_{l,\kappa}, n_{d,\kappa}, c_{\kappa}) \\ &+\partial_{\xi} c \partial_{c} \psi(\zeta, n_{l,\kappa}, n_{d,\kappa}, c_{\kappa})+\int_{\tau}^{t} \int_{Q} \partial_{\xi}(H(n_{\kappa})) \psi(n_{\kappa}, n_{l,\kappa}, n_{d,\kappa}, c_{\kappa}),
\end{aligned}
$$

it follows that 
$$
\begin{aligned}
& \int_{\tau}^{t} \int_{Q} \chi_{\kappa}(n_{\kappa})(\partial_{\xi} n_{\kappa})^{2}+\frac{1}{2} \int_{Q}|\chi_{\kappa}(n_{\kappa}) \nabla n_{\kappa}|^{2}|_{\xi=t} \\
=&\frac{1}{2} \int_{Q}|\chi_{\kappa}(n_{\kappa}) \nabla n_{\kappa}|^{2}|_{\xi=\tau}+\int_{Q} \Psi(n_{\kappa}, n_{l,\kappa}, n_{d,\kappa}, c_{\kappa})|_{\xi=t}-\int_{Q} \Psi(n_{\kappa}, n_{l,\kappa}, n_{d,\kappa}, c_{\kappa})|_{\xi=\tau} \\ &+ \int_{\tau}^{t} \int_{Q} \int_{0}^{n_{\kappa}} \chi_{\kappa}(\zeta)(\partial_{\xi} n_l \partial_{n_l} \psi(\zeta, n_{l,\kappa}, n_{d,\kappa}, c_{\kappa})+\partial_{\xi} n_{d} \partial_{n_d} \psi(\zeta, n_{l,\kappa}, n_{d,\kappa}, c_{\kappa}) \\ &+\partial_{\xi} c \partial_{c} \psi(\zeta, n_{l,\kappa}, n_{d,\kappa}, c_{\kappa})
\end{aligned}
$$
Lemma \ref{nbound} and \eqref{Q8} imply that there exists $\kappa_{0}>0$ such that  $n_{\kappa}$ satisfy $n_{\kappa}\leq n_{M} <1$ in $Q_{T}$ for all $0<\kappa<\kappa_{0}$. Consequently, $\chi_{\kappa}(n_{\kappa})$ is positive and uniformly bounded from above by a constant which is independent from $\kappa$ i.e 
\begin{equation}\label{SD}
  0 < \chi_{\kappa}\left(n_{\kappa}(t, x)\right)=\kappa n_{\kappa}(t, x)(1-n_{\kappa}(t, x))^{-2} \leq \kappa n_M(1-n_M)^{-2} =\frac{p(n_{M})}{(1- n_{M})} \leq C, \qquad (t, x) \in Q_{T}.
\end{equation}
Applying Young's inequality on the last integral we can estimate
\begin{align}
&\int_{\tau}^{t} \int_{Q} \int_{0}^{n_{\kappa}} \chi_{\kappa}(\zeta)(\partial_{\xi} n_l \partial_{n_l} \psi(\zeta, n_{l,\kappa}, n_{d,\kappa}, c_{\kappa})+\partial_{\xi} n_{d} \partial_{n_d} \psi(\zeta, n_{l,\kappa}, n_{d,\kappa}, c_{\kappa}) +\partial_{\xi} c \partial_{c} \psi(\zeta, n_{l,\kappa}, n_{d,\kappa}, c_{\kappa})\\
&\leq C \int_{\tau}^{t} \int_{Q}\left((|\partial_{\xi} n_{l, \kappa}|^{2}+|\partial_{\xi} n_{\kappa, d}|^{2} + |\partial_{\xi} c|^{2})+ \int_{0} ^{n_{\kappa}} \sum_{i=2}^{4}|\partial_{i} \psi( \zeta, n_{\kappa, l}, n_{\kappa, d}, c_{\kappa})|^{2}\right),
\end{align}
estimates \eqref{EQ3},\eqref{EQ1},\eqref{EQ2}, and \eqref{SD} and setting $\tau =0$, lead to 
\begin{equation*}
\int_{Q_{T}} \chi_{\kappa}(n_{\kappa})(\partial_{t} n_{\kappa})^{2} \leq C', \quad \sup _{t \in[0, T]} \int_{Q}|\chi_{\kappa}(n_{\kappa}) \nabla n_{\kappa}|^{2} \leq C'.
\end{equation*}
since $n_{\kappa}$ is uniformly bounded, we have
\begin{align*}
\int_{Q_{T}} \left( \partial_{t} H_{\kappa}(n_{\kappa})\right)^2 = \int_{Q_T}(\chi_{\kappa}(n_{\kappa})\partial_{t} n_{\kappa})^{2}  \leq \|\chi_{\kappa}(n_{\kappa})\|_{L^\infty(Q_T)}\int_{ Q_{T}}\chi_{\kappa}(n_{\kappa})(\partial_{t} n_{\kappa})^{2} \leq C'
\end{align*}
which implies
\begin{equation}
 \int_{Q_{T}} \left( \partial_{t} H_{\kappa}(n_{\kappa})\right)^2  \leq C', \qquad \sup _{t \in[0, T]} \int_{Q} |\nabla H_{\kappa}(n_{\kappa})|^{2} \leq C' .
\end{equation}
This shows that the family $\Gamma_{\kappa} = H_{\kappa}(n_\kappa), \  \text{for} \  0<\kappa<\kappa_{1}$, is uniformly bounded in $W=\{u \in L^{\infty}(0, T ; H^{1}(Q)) \mid \partial_{t} u \in L^{2}(0, T ; L^{2}(Q))\}$, which is compactly embedded into $C\left(0, T ; L^{2}(Q)\right)$ by Aubin-Lions' Lemma. Consequently, there exists $\Gamma_0 \in C\left(0, T ; L^{2}(Q)\right)$ such that 
\[
\Gamma_{\kappa} \to \Gamma_0 \qquad in \quad C(0, T; L^{2}(Q)) \quad as \ \ \kappa \to 0 
\]
We have from the definition of $H(n_{\kappa})= p_{\kappa}- \kappa\ln(1-n_\kappa)$, which implies that $p_{0}=\Gamma_0$, therefore \eqref{EQ10} holds. \\
\end{proof} 
The next proposition shows the energy estimate, we introduce 
the quantity 
 \[
 \Phi_{\kappa}(n)= \int_{0}^n H_{\kappa}(s) ds.
 \]
 
\begin{prop} (Energy estimate )
Let the above assumptions hold. Then, the following estimate is  valid 
\begin{equation}
    \|\Phi_{\kappa}(n_{\kappa})\|_{L^{\infty}(0,T; L^1(Q))}+ \|\nabla H_{\kappa}(n_{\kappa})\|^2_{L^2(Q_T)} \leq C
\end{equation}
where $C$ is a constant independent of $\kappa.$
\end{prop} 
\begin{proof}
We multiply the equation \eqref{Q1} by $H_{\kappa}(n_{\kappa})$ and integrate over space, we have 
\begin{align}
    \frac{d}{dt}\int_{Q} \Phi_{\kappa}(n_{\kappa}) dx+ \int_{Q} |\nabla H_{\kappa}(n_{\kappa})|^2 dx = \int_{Q} (G(c_{\kappa}) n_{l, \kappa} - \mu n_{d, \kappa})) H_{\kappa}(n_{\kappa}) dx
\end{align}
inserting the expression of $H_{\kappa}(n_{\kappa})$, we get
\begin{align}
    \frac{d}{dt}\int_{Q} \Phi_{\kappa}(n_{\kappa}) dx+ \int_{Q} |\nabla H_{\kappa}(n_{\kappa})|^2 dx &\leq (G_m +\mu)\int_{Q} n_{\kappa} (p_{\kappa} + \kappa\ln (1-n_\kappa)) \\
    &\leq (G_m +\mu)\left(\int_{Q} p_\kappa -\kappa n_{\kappa} + \kappa  \int_{Q}n_{\kappa} \ln (1- n_{\kappa})\right)
\end{align}
after time integrating, we get 
\begin{align}
    \int_{Q} \Phi_{\kappa}(n_{\kappa}) dx+ \int_{Q_T} |\nabla H_{\kappa}(n_{\kappa})|^2 dxdt &\leq (G_m +\mu)\int_{Q_T} n_{\kappa} (p_{\kappa} + \kappa\ln (1-n_\kappa)) \\
    &\leq (G_m +\mu)\left(\int_{Q_T} p_\kappa -\kappa n_{\kappa} dx dt + \kappa  \int_{Q_T}n_{\kappa} \ln (1- n_{\kappa} dx dt)\right) + \int_Q \Phi(n^0)dx .
\end{align}

Taking into account (H6), $n_\kappa$ is bounded in $L^1$, and $x \to x \ln (1-x)$ is uniformly bounded on $[0, n_M]$, thus the right hand side integrals are uniformly bounded. \\
\textbf{$L^2$ Estimate for pressure gradient:}
We may write  
\[
\int_{Q_T} |\nabla H_{\kappa}(n_{\kappa})|^2 dxdt = \int_{Q_T} n_{\kappa}^2 |\nabla p_{\kappa}|^2dx dt \leq C,
\]
given that $n_{\kappa}$ is uniformly bounded in $L^{\infty}$,  we get
\begin{equation}\label{Pr}
\int_{Q_T} |\nabla p_{\kappa}|^2dx dt \leq C,
\end{equation}
taking Friedrich's inequality into account we deduce that $p_{\kappa} \in L^{2}(Q_T).$ Consequently, the sequence $(\nabla p_{\kappa})_{\kappa}$ converges in the strong topology of $L^2( 0,T; L^2(Q))$ to $\nabla p_{0}$ and thus \eqref{EQ12} holds.
\end{proof}
\section{Complementarity relation}
 The usual strategy to prove the complementarity relation is to prove the strong convergence of the gradient pressure $\nabla p_{\kappa}$, and pass to the limit in the equation for the pressure \eqref{eq-pressur}.  Since we have no control on $\partial_t p_{\kappa}$, we will be able to detect the limit only after the proof of the strong
compactness of $\nabla H_{\kappa}$, thanks to the uniform estimates established in the previous section. More precisely we prove the following
complementarity relation

\begin{equation}
-p_{0}^{2}\Delta p_{0}= p_{0}^2 (G(c_0)n_{l,0} -\mu n_{d,0}).
\end{equation}
We know that 
\begin{align}
\Delta H_{\kappa}(n_{\kappa})&= \partial_t n_{\kappa} -(G(c_{\kappa}) n_{l, \kappa}-\mu n_{d, \kappa}),  \label{Eqqq}
\end{align}
and
\[
 \Delta p_{\kappa} = \Delta H_{\kappa} -\kappa \Delta \ln (1-n_{\kappa}).
\]
 From \eqref{Cond1}, \eqref{EqL}-\eqref{EqD}, and  the inequality $\ln(1-n_{\kappa}) \leq \ln(1-n_{M}) \leq 0$ we deduce that $\Delta H_{\kappa} $ and $\Delta p_{\kappa}$ are bounded in $L^{\infty}(0,T; L^1(Q))$,  therefore we have local compactness in space for $\nabla H_{\kappa}$. To ensure time compactness we apply the Aubin-Lions lemma. We may write  
\begin{align*}
\partial_t(\nabla H_{\kappa})&= \nabla(\partial_t H_{\kappa}) \\
&= \nabla [\partial_{t} (p_{\kappa}+ \kappa \ln(1-n_{\kappa}) )]\\
&=\nabla[\frac{p_{\kappa}}{p_{\kappa} +\kappa} \partial_t p_{\kappa}]\\
&=\nabla[\frac{p_{\kappa}}{\kappa}((p_{\kappa}+\kappa) R_{\kappa}+ p_{\kappa} \Delta p_{\kappa} )+\frac{p_{\kappa}}{p_{\kappa}+\kappa} |\nabla p_{\kappa}|^2 ],
\end{align*} 
 from \eqref{EqN}, \eqref{Pr} and \eqref{Eqqq} $\partial_{t} H_{\kappa}$, $|\nabla p_{\kappa}|^2$, and $\Delta p_{\kappa}$ are bounded functions in $L^1$, on the other hand, given that $n_{\kappa}$, and $ p_{\kappa}$  are uniformly bounded in $L^{\infty}$, we deduce that the r.h.s term is a sum of space derivatives of functions bounded in $L^1$. Consequently, we can extract a sub-sequence such that
\[
\nabla H_{\kappa} \to \nabla p_{0} \qquad \text{strongly in} \  L^1(Q_T).
\]
After extraction of a sub-sequence we obtain convergence almost everywhere for $\nabla H_{\kappa}$.\\
Let $\varphi \in \mathcal{D}(Q_T)$ be a test function. We consider the equation \eqref{Eqqq}, and multiply it by $\kappa P'_{\kappa}$ we obtain
\[
(p_{\kappa}+ \kappa)^2\Delta H_{\kappa}= \frac{\kappa^2}{(1-n_{\kappa})^2} \partial_t n_{\kappa} - (p_{\kappa}+ \kappa)^2 (G(c_{\kappa}) n_{l, \kappa}-\mu n_{d, \kappa}),
\]
we multiply it by $\varphi$  and we integrate over $Q_T$
\begin{align*}
    \kappa^2 \int_{Q_T} (\Delta H_{\kappa} -\frac{1}{(1-n_{\kappa})^2} \partial_t n_{\kappa} + R_{\kappa}) \varphi &+ 2\kappa \int_{Q_T} (\Delta H_{\kappa} -R_{\kappa})p_{\kappa} \varphi \\
    &=  \int_{Q_T} 2 p_{\kappa} \varphi \nabla p_{\kappa} \nabla H_{\kappa}  + p^2_{\kappa} \nabla H_{\kappa} \nabla \varphi -  p^2_{\kappa} R_{\kappa} \varphi.
\end{align*}
Thanks to the bounds provided the previous section. From \eqref{EqC}, \eqref{EqL}, \eqref{EqD}, and \eqref{EqN},  $c_{\kappa}, n_{l, \kappa}, n_{d, \kappa}$, and $H_{\kappa}$ are compactly embedded in $C(0,T; L^2(Q))$ which is continuously embedded in $C(0,T; L^1(Q))$,  $p_{\kappa}$ is uniformly bounded in $L^2(Q_T)$ and thus, after the extraction of sub-sequences, we can pass to the limit for $\kappa \to 0$ in the product and obtain the complementarity relation 
\[
\int_{Q_T} \left(2 p_{0} |\nabla p_{0}|^2 \varphi   + p^2_{0} \nabla p_{0} \nabla \varphi -  p^2_{0} R_{\kappa} \varphi \right) dx dt,
\]
which is equivalent to
 \[
-\Delta p_0 = G(c_0)n_{l,0}-\mu n_{d,0}, \qquad  \mathcal{D'}(Q_T).
 \]


\section{Convergence rate: proof of Theorem \ref{Theo-conv-rate}}
 Recently, the authors of the paper \cite{David}  proved  the $L^{4/3}$-convergence rate using interpolation with BV bound. Here we prove the result in $L^1$ directly without using interpolation  with BV bound.
We define the function $\varphi$ to be the solution of the
following parabolic equation in $(0,T) \times Q$
\begin{equation}\label{EQ4}
\begin{cases}
\partial_t \varphi =(\lambda(t)+\varepsilon)\Delta \varphi \\
\varphi(t,0)=0 \\
\varphi(0,x)=\varphi_0,
\end{cases}
\end{equation}
where $\varepsilon$ is a small regularization parameter, and the function $\lambda$ is defined by
\begin{equation}\label{lamda}
 \lambda(t,x):= \int_{0}^1 H'(\xi n_{\kappa}+(1-\xi)n_{\kappa'})d\xi.
\end{equation}
for $0< \kappa' < \kappa$, and $H$ is defined by \eqref{H-equation}.  Our approach is based on the following Lemma which appeared in the investigation of Lipschitz semigroupe continuous for bacterial biofilm model \cite{MA}
\begin{lemma}\label{lemma2}
The problem \eqref{EQ4} has a unique solution $\varphi$ which satisfies 

$$\varphi \in L^{\infty}((0,T)\times Q) \cap L^{\infty}(0,T; W^{1,2}(Q)), 
 \quad \Delta \varphi \in L^{2}((0,T)\times Q).$$
Moreover, the solution $\varphi$ is subject to the inequalities

\begin{gather}
 \|\varphi(t)\|_{L^{\infty}(Q)} \leq \|\varphi_{0}\|_{L^{\infty}(Q)}, \label{EQ5'} \\
 \|\varphi(t)\|_{W^{1,2}(Q)}^{2}+ C \varepsilon \int_{0}^{t}\left\|\Delta \varphi(\xi)\right\|_{L^{2}(Q)}^{2} d \xi \leq\left\|\varphi_{0}\right\|_{W^{1,2}(Q)}^{2}, \label{EQ5}
\end{gather}
\end{lemma}
\begin{proof}
We have from the previous section $(n_{\kappa}, c_{\kappa}, n_{\kappa'}, c_{\kappa'}) \in \left(C(0,T; L^2(Q))\cap L^{\infty}((0,T)\times Q)\right)^4$, then the function $\lambda$ satisfies

\begin{equation}\label{EQQ6}
\lambda \in L^{\infty}( (0,T)\times Q ),  \quad \text { and } \ \ \lambda(t, x) \geq 0.
\end{equation}
Therefore the problem \eqref{EQ4} with $\varepsilon >0$ is a non-degenerate parabolic equation. To drive the estimates for $\varphi$ we use the maximum principle. We multiply equation \eqref{EQ4} by $|\varphi - \|\varphi_0\|_{\infty}|^{+}$, and integrate in space  we get 
\begin{equation}
\frac{1}{2} \frac{d}{dt} \int_{Q}(|\varphi -\|\varphi_0\|_{\infty}|^{+})^2 \leq 0,
\end{equation}
after time integration we get  
\begin{equation}
\frac{1}{2} \int_{Q}(|\varphi(t) -\|\varphi_0\|_{\infty}|^{+})^2 \leq 0.
\end{equation}
So $|\varphi - \|\varphi_0\|_{L^\infty}|^{+} =0$, thus we get  $\|\varphi \|_{L^\infty(Q)} \leq \|\varphi_0\|_{L^\infty(Q)}.$
To drive the estimate \eqref{EQ5}, we multiply equation \eqref{EQ4} by $\Delta \varphi$ and integrate in space 
\begin{equation}
\int_{Q} \partial_t|\nabla \varphi |^2 + \int_{Q} (\lambda (t)+\varepsilon )|\Delta \varphi|^2=0,
\end{equation} 
we integrate in time, and considering \eqref{EQQ6} we get
\begin{equation}
\int_{Q} |\nabla \varphi(t) |^2 + C \varepsilon \int_{0}^t\int_{Q}|\Delta \varphi|^2 \leq \int_{Q} |\nabla \varphi (0)|^2,
\end{equation}
thus \eqref{EQ5}. We omit here the investigation of the existence and uniqueness, for the reader reference's see \cite{Lad}. This finishes the proof of Lemma \ref{lemma2}.  
\end{proof}
Now we proceed to Proof Theorem \ref{Theo-conv-rate}\\
\textbf{Proof of Theorem \ref{Theo-conv-rate}}
 
 We consider  $W(t)=n_{\kappa}(t)-n_{\kappa'}(t)$, and $Z(t)=c_{\kappa}(t)-c_{\kappa'}(t)$, for $0<\kappa'< \kappa$. Then these functions satisfy the following equations  
\begin{equation}\label{EQ6}
\partial_t(W(t))- \Delta (\lambda(t)W(t))= R(t)= G(c_{\kappa}- c_{\kappa'})(n_{l,\kappa}-n_{l,\kappa'})-\mu (n_{d,\kappa}-n_{d,\kappa'}),
\end{equation}
and 
\begin{equation}\label{EQ7}
\partial_t(Z(t))- d\Delta (Z(t))= R_3(t)= -f(c_{\kappa}- c_{\kappa'})(n_{l, \kappa}-n_{l, \kappa'}).
\end{equation}
Let $\varphi_0 \in C^\infty_0(Q)$ be an arbitrary function,  and $\varphi$ be a solution of the problem \eqref{EQ4}. We test equation  \eqref{EQ6} against  $\varphi(t)$ and integrate in space  and time we obtain
\begin{align*}
\int_{Q} W(t) \varphi(t)dx- \int_{Q} W(0)\varphi(0)dx & + \int_0^t\int_{Q} \left( \partial_t \varphi -\lambda(\xi)\Delta \varphi \right) W(\xi) d\xi dx \\
&= \int_0^t\int_{Q} \varepsilon \Delta \varphi  W(\xi)d\xi dx + \int_{0}^t\int_Q R(\xi)\varphi d\xi dx, 
\end{align*} 
where we used the fact that $\varphi(t)$ solves \eqref{EQ4}. Using Young's inequality and estimate \eqref{EQ5'}

\begin{align*}
\int_{Q} W(t) \phi(t) &\leq  \sqrt{\varepsilon} \left( \int_0^t\left(\|W(t)\|_{L^2(Q)}^2  + \varepsilon \|\Delta \varphi\|^2_{L^2(Q)}\right) d\xi  \right)\\
& + \|\varphi_0\|_{L^\infty(Q)} \left( \|W(0)\|_{L^1(Q)} + \int_{0}^t \|R(\xi)\|_{L^1(Q)} d\xi\right).
\end{align*}
Now taking the limit $\varepsilon \to 0^+$ and using estimate \eqref{EQ5'}, we derive
\begin{equation}\label{EQ8}
\int_{Q} W(t) \phi(t) \leq \|\varphi_0\|_{L^\infty(Q)} \left( \|W(0)\|_{L^1(Q)} + \int_{0}^t \|R(\xi)\|_{L^1(Q)} d\xi\right),
\end{equation}
for every $\varphi_0 \in C^\infty_0(Q)$. Taking now any function  $\varphi _0\in L^\infty(Q)$ and approximate it by a sequence $\varphi_{0}^n\in C^{\infty}_0(Q)$ such that
 $\|\varphi _{0}^n\|_{\infty}\leq \|\varphi_0\|_{L^\infty(Q)}$ and $\|\varphi_{0}^n -\varphi_0\|_{L^1} \to 0 $ as $n\to \infty$
thus \eqref{EQ8} is valid for every $\varphi_0\in L^\infty(Q)$. Therefore we get 
\begin{equation}
\|W(t)\|_{L^1(Q)} \leq  \|W(0)\|_{L^1(Q)} + \int_{0}^t \|R(\xi)\|_{L^1(Q)} d\xi,
\end{equation}
which implies 
\begin{equation}
\|n_{\kappa}(t) - n_{\kappa'}(t)\|_{L^1(Q)} \leq \|n_{\kappa}(0)-n_{\kappa'}(0)\|_{L^1(Q)}+ C,
\end{equation}
from \eqref{EQ9}, $n_{\kappa}$  converge in the strong topology of $C(0,T; L^1(Q))$ to $n_{0}$ (due to the continuous embedding $C(0,T; L^2(Q))\hookrightarrow C(0,T; L^1(Q))$). Consequently, taking $\kappa' \to 0$ we deduce the following rate of the convergence in $L^1(Q)$
\begin{equation}
\|n_{\kappa}(t)-n_{0}(t)\|_{L^1(Q)} \leq \|n_{\kappa}(0)-n_{0}(0)\|_{L^1(Q)} +C,
\end{equation}
where $C$ is a positive constant defined as 
\begin{equation}
C=G_m t \|n_{\kappa}\|_{L^1(Q)}.
\end{equation}
The second estimate of the convergence rate \eqref{Conv-c} can be proved analogously.

\bibliography{sample}

\end{document}